\documentclass[letterpaper, 10 pt, conference]{ieeeconf}  

\IEEEoverridecommandlockouts                              
\usepackage{amssymb}

\usepackage[pdftex]{graphicx}
\usepackage{tikz}
\usepackage{amsmath}
\usepackage{framed}
\usepackage{bm}
\usepackage{cite}
\usepackage{comment}
\usepackage{algorithmic}
\usepackage{algorithm}
\usepackage{url}
\usepackage{booktabs}
\usepackage{array}

\usetikzlibrary{intersections,calc,arrows.meta}

\usepackage[sort]{cleveref}
\usepackage{mathtools}
\usepackage{subcaption}
\usepackage{multirow}
\usepackage{threeparttable}

\usepackage{pgfplots}

\usepackage{mathtools}
\mathtoolsset{showonlyrefs=true} 

\usepackage{amsfonts}
\usepackage{color}
\usepackage{cases}
\usepackage{enumerate}

\newtheorem{theorem}{Theorem}
\newtheorem{assumption}{Assumption}

\newtheorem{remark}{Remark}
\newtheorem{lemma}{Lemma}
\newtheorem{example}{Example}
\newtheorem{corollary}{Corollary}

\usepackage{blkarray}

\DeclareMathOperator{\diag}{diag}
\DeclareMathOperator{\tr}{tr}

\newcommand{\Real}{\mathbb{R}}

\newcommand{\bb}{\mathbb}

\title{\LARGE \bf
Task-Dependent Weighted Average Energy Controllability Score\\ for Network Intervention
}

\author{Kazuhiro Sato
\thanks{This work was supported by JST PRESTO, Japan, Grant Number JPMJPR25K4. }
\thanks{K. Sato is with Department of Mathematical Informatics, Graduate School of Information Science and Technology, The University of Tokyo, Tokyo 113-8656, Japan.
        {\tt\small kazuhiro@mist.i.u-tokyo.ac.jp}}
}

\begin{document}

\maketitle
\thispagestyle{empty}
\pagestyle{empty}

\begin{abstract}
Controllability scores provide principled information on where intervention should be applied in large-scale network systems when explicit control design is difficult.
Two representative controllability scores are the volumetric controllability score (VCS) and the average energy controllability score (AECS).
While both are important, the standard AECS treats all state-transition directions uniformly. In this paper, we propose the weighted average energy controllability score (W-AECS), a task-dependent extension of AECS that incorporates a prescribed transition of interest through a weighting matrix.
We show that the proposed formulation admits a control-theoretic interpretation via expected minimum-energy steering, and establish strict convexity and generic uniqueness.
These results support the interpretation of W-AECS as a well-defined node-wise task-dependent intervention score.
We also illustrate the proposed method on a structural brain-network dataset, where transition-dependent weighting reshapes the scoring pattern, yielding a VCS-like preference among the highest-ranked regions while preserving an overall structure distinct from both standard AECS and VCS.
\end{abstract}


\section{Introduction}

Large-scale network systems often come with control objectives of clear practical importance, while the explicit design of effective control inputs remains extremely difficult. Such situations arise naturally in many real applications. In brain networks, one may wish to alleviate symptoms or restore desirable functional states \cite{gu2015brain, luppi2024cognitive}; in social networks, one may wish to induce behavioral change or promote socially beneficial outcomes \cite{aral2017exercise}. Although these objectives may be clear at a conceptual level, designing input signals that reliably realize them is often beyond practical reach.

One major reason is that the available mathematical model is typically only a coarse approximation of the true system \cite{thompson2022escape}. Brain networks are governed by complex physical and biological mechanisms that are not yet fully understood \cite{bahrami2023brain}, while social networks involve heterogeneous interactions, strategic behavior, and context-dependent effects that are even harder to capture accurately \cite{fatemi2023network}. One might hope to bypass detailed modeling through purely data-driven input design, but this is often unrealistic in practice. In systems such as brain networks and social networks, the repeated large-scale interventions needed for reliable design or validation are frequently limited by ethical, practical, and inferential constraints.

Nevertheless, decision making cannot be deferred indefinitely \cite{cao2026cphs}. Even when the exact time-varying input needed to achieve a control objective is unknown, it is still valuable to know where intervention should be concentrated and how strongly intervention resources should be allocated. Such information can already support meaningful action. For example, in brain networks, intervention locations may correspond to stimulation sites and intervention intensities to stimulation frequency or dosage \cite{gu2015brain, luppi2024cognitive}; in social systems, they may correspond to individuals or regions and to the level of budget or resources assigned to them \cite{yanchenko2024influence}.

Controllability scores are intended for situations in which exact control synthesis is out of reach, but meaningful intervention decisions still have to be made \cite{sato2022controllability, sato2025uniqueness}. Rather than prescribing a detailed control law, they provide principled controllability-based information on where intervention should be applied and how strongly available resources should be allocated. Concretely, one introduces a virtual allocation vector over candidate intervention nodes and solves an optimization problem based on a controllability Gramian; the resulting optimizer is then interpreted as a node-wise controllability score. In the existing framework, two representative Gramian-based objectives have been considered: a volumetric objective, which favors broad steerability in state space, and an average-energy objective, which favors low minimum control energy on average. The optimal allocations obtained from these two objectives are called the volumetric controllability score (VCS) and the average energy controllability score (AECS), respectively. Both VCS and AECS are important controllability scores that capture different aspects of controllability.
The present paper focuses on AECS, whose objective is directly tied to minimum control energy.
However, this objective is intrinsically isotropic: it treats all state-space directions uniformly and does not distinguish their practical importance.

To address this limitation, this paper develops W-AECS, a weighted extension of AECS derived from finite-horizon minimum-energy steering. The basic observation is that the control effort required to realize a state transition depends on the discrepancy between the desired terminal state and the state reached by the uncontrolled dynamics. If the initial and terminal states are modeled probabilistically, then averaging the corresponding minimum control energy naturally yields a weighted controllability criterion determined by the second-order geometry of the relevant transitions. Hence, the proposed formulation is not an ad hoc reweighting of AECS, but a principled generalization grounded in finite-horizon control. It recovers the standard AECS as a special isotropic case, while allowing anisotropic and task-dependent transition structures to be incorporated systematically. Because the formulation is inherently finite-horizon, it remains meaningful even for marginally stable or unstable systems, and is therefore well suited to intervention problems involving a distribution of practically relevant transitions rather than a single prescribed target state.

The main contributions of this paper are as follows.
\begin{enumerate}
\item We derive a finite-horizon weighted extension of AECS from expected minimum-energy steering, thereby giving a control-theoretic interpretation to the proposed task-dependent criterion. Importantly, the formulation requires only the mean and covariance of the induced state displacement, so that a complete joint probabilistic model of the initial and terminal states is unnecessary.

\item We establish the optimization-theoretic foundations of the proposed formulation, including directional derivative formulas and sufficient conditions for strict convexity with respect to the virtual-actuation allocation.
We further prove a generic uniqueness result: for any system matrix, the optimizer is unique for almost all time horizons.
This implies that W-AECS is generically well defined.
Hence, W-AECS can be interpreted as a meaningful task-dependent centrality measure, since each node receives a uniquely determined intervention-importance value.
\item We demonstrate the proposed W-AECS on a structural brain-network dataset and show that task-dependent weighting can substantially reshape intervention scoring.
By constructing the weighting matrix from a transition toward a VCS-favored target state, we show that W-AECS can reproduce a VCS-like preference among the highest-ranked ROIs, while still retaining an overall scoring structure distinct from both the standard AECS and VCS.
In addition, W-AECS exhibits a pronounced separation between the highest- and lowest-ranked ROIs, a feature not observed for VCS or AECS in the corresponding brain-network experiments reported in \cite[Section~IV]{sato2025uniqueness}.
\end{enumerate}

The rest of this paper is organized as follows.
Section~II introduces the virtual actuation framework and derives the proposed weighted average energy objective from expected minimum-energy steering.
Section~III defines W-AECS and studies its basic properties, including existence, strict convexity, generic uniqueness, and a projected-gradient computation scheme.
Section~IV reports a numerical experiment on a structural brain-network dataset and examines how task-dependent weighting affects the resulting scoring in comparison with AECS and VCS.
Section~V concludes the paper.

{\bf Notations:}
Let $\bb{R}$ denote the set of real numbers.
The symbol $O$ denotes the zero matrix of appropriate size.
For a matrix $X$, $X^\top$ denotes the transpose.
For a square matrix $X$, $\tr(X)$, $\det(X)$, and $\exp(X)$ denote the trace, determinant, and matrix exponential of $X$, respectively.
For any symmetric matrix \(X\), we write \(X\succeq O\) (resp., \(X\succ O\)) to denote that \(X\) is positive semidefinite (resp., positive definite), and \(\lambda_{\min}(X)\) and \(\lambda_{\max}(X)\) denote its minimum and maximum eigenvalues, respectively.
For a vector $x$, $x_i$ denotes its $i$th component.
The $n\times n$ identity matrix is denoted by $I_n$.
The symbol $\bm{1}$ denotes all-ones vector in $\Real^n$, and $e_i$ is the $i$th standard basis vector.
The symbol $\diag(a_1,\ldots, a_n)$ denotes the diagonal matrix with diagonal entries $a_1,\dots,a_n$.
For $T>0$, $L^2([0,T];\mathbb{R}^n)$ denotes the space of square-integrable $\mathbb{R}^n$-valued functions on $[0,T]$.
The notation $\mathcal{N}(\mu,\Sigma)$ denotes the Gaussian distribution with mean $\mu$ and covariance $\Sigma$.
Thus, $x\sim\mathcal{N}(\mu,\Sigma)$ means that $x$ is a Gaussian random vector with mean $\mu$ and covariance $\Sigma$.
For a differentiable mapping \(F\), \({\rm D}F(x)[h]\) denotes the directional derivative of \(F\) at \(x\) along the direction \(h\).

\section{Preliminaries}

We consider an autonomous linear time-invariant network system
\begin{equation}
\dot{x}(t)=Ax(t),
\quad
A\in\mathbb{R}^{n\times n},
\label{eq:autonomous}
\end{equation}
where
$x(t)=\begin{pmatrix}x_1(t)&\cdots&x_n(t)\end{pmatrix}^\top\in\mathbb{R}^n$
is the state vector, and $x_i(t)$ denotes the state associated with node $i$.
The matrix $A$ encodes the weighted interactions among nodes.
Since autonomous system \eqref{eq:autonomous} has no physical input channel, the usual notion of controllability is not directly applicable. However, our goal is not to design an actual control law for \eqref{eq:autonomous}, but to quantify \emph{where} intervention resources should be allocated and \emph{how strongly} they should be allocated. To this end, we introduce virtual input channels at all state nodes and use the resulting controllability Gramian.

\subsection{Virtual actuation model}

To quantify intervention importance, we augment \eqref{eq:autonomous} with virtual inputs and consider the system
\begin{equation}
\dot{x}(t)=Ax(t)+B(p)u(t),
\label{eq:virtual_system}
\end{equation}
where 
\[
p\in
\Delta_n:=\left\{p\in\mathbb{R}^n\mid {\bf 1}^\top p=1,\,\,p_i\geq 0\,\,(i=1,\ldots, n)\right\}
\]
is a resource-allocation vector on the simplex, and
\begin{align}
    B(p) := \mathrm{diag}(\sqrt{p_1},\dots,\sqrt{p_n}). \label{diagB}
\end{align}
Under this construction, each state node is associated with a corresponding virtual input channel, and the scalar $p_i$ specifies the amount of virtual actuation resource allocated to node $i$. In this sense, the vector $p$ represents an intervention-allocation profile over candidate nodes. This interpretation is illustrated in \cite[Fig.~2]{sato2025uniqueness}.

\subsection{Finite-horizon controllability Gramian}

For a fixed horizon $T>0$, the finite-horizon controllability Gramian associated with \eqref{eq:virtual_system} is defined by
\begin{equation}
W(p,T)
=
\int_0^T \exp(At)B(p)B(p)^\top \exp(A^\top t)\,{\rm d}t.
\label{eq:Wp_def}
\end{equation}
Using the diagonal form of $B(p)$ in \eqref{diagB}, we obtain
\begin{equation}
W(p,T)=\sum_{i=1}^n p_i W_i(T),
\label{eq:Wp_linear}
\end{equation}
where
\begin{equation}
W_i(T)
:=
\int_0^T \exp(At)e_ie_i^\top \exp(A^\top t)\,{\rm d}t.
\label{eq:Wi_def}
\end{equation}
Thus, $W_i(T)$ is the finite-horizon controllability Gramian corresponding to a single virtual input channel acting only on node $i$, while $W(p,T)$ is the weighted sum of these node-wise Gramians according to the allocation profile $p$.
The representation \eqref{eq:Wp_linear} is fundamental to controllability scoring. In particular, the controllability Gramian depends linearly on the allocation vector $p$, which enables the definition of optimization-based controllability scores through suitable scalar criteria applied to $W(p,T)$ \cite{sato2022controllability}.

\subsection{Expected Minimum Energy}

We now derive the weighted average energy controllability objective from the finite-horizon minimum-energy steering problem associated with virtually actuated system \eqref{eq:virtual_system}.
Throughout this subsection, we fix a horizon $T>0$, and assume that $W(p,T)\succ O$, which is equivalent to that system \eqref{eq:virtual_system} is controllable \cite{kalman1960contributions}.

As shown in \cite{kalman1960contributions} and \cite[Theorem~8.4]{lygeros_notes},
 for LTI system \eqref{eq:virtual_system} with virtual input $u(\cdot) \in L^2([0,T];\mathbb{R}^n)$, the minimum input energy required to steer the state from $x_0$ to $x_T$ over $[0,T]$ is
\begin{align}
E_{\min}(x_T\mid x_0;p,T) \label{eq:min_energy}
=
z(T)^\top
W(p,T)^{-1}
z(T),
\end{align}
where
\begin{equation}
z(T):=x_T-\exp(AT)x_0.
\label{eq:z_def}
\end{equation}

To capture which state displacements are regarded as important in the intervention task, we introduce the matrix
\begin{equation}
M(T):=\mathbb{E}\left[z(T)z(T)^\top\right],
\label{eq:M_def}
\end{equation}
assuming that this second moment is well defined.
Here, $\mathbb{E}[\cdot]$ denotes expectation with respect to a prescribed probability distribution of the displacement vector $z(T)$.
This distribution need not be understood as an exact statistical description of the underlying system; rather, especially under coarse dynamical modeling, it serves as a task-dependent approximation that captures the state displacements most relevant to the intervention analysis.
Hence, it is not arbitrary, but should be chosen consistently with available knowledge of the system and the intended control objective.
When $z(T)$ is defined through \eqref{eq:z_def}, such a distribution may be induced from a probabilistic model of $(x_0,x_T)$ reflecting relevant initial states, desirable terminal states, and an assumed dependence structure between them.
The matrix $M(T)$ then serves as a weighting matrix tailored to the intervention task.

Accordingly, we define the expected minimum energy by
\begin{align}
J(p;T)
&:=
\mathbb{E} \left[E_{\min}(x_T\mid x_0;p,T)\right] \\
&= \mathbb{E} \left[z(T)^\top W(p,T)^{-1}z(T)\right] \\
&= \mathrm{tr} \left(W(p,T)^{-1}M(T)\right)
\label{eq:J_def}
\end{align}
for all $p\in\Delta_n$ such that $W(p,T)\succ O$.
Using the definition of the covariance matrix
\[
\mathrm{Cov}(z(T))
:=
\mathbb{E} \left[
\bigl(z(T)-\mathbb{E}[z(T)]\bigr)
\bigl(z(T)-\mathbb{E}[z(T)]\bigr)^\top
\right],
\]
the weighting matrix $M(T)$ in \eqref{eq:M_def} can be expressed as
\begin{align}
M(T)
=
\mathrm{Cov}(z(T))
+
\mathbb{E}[z(T)]\mathbb{E}[z(T)]^\top.
\label{eq_M}
\end{align}
Thus, $J(p;T)$ in \eqref{eq:J_def} can be rewritten as
\begin{align}
    J(p;T) = &{\rm tr} \left(W(p,T)^{-1} {\rm Cov}(z(T))\right) \label{eq:J_decomp} \\
    &+{\bb E}[z(T)]^\top W(p,T)^{-1} {\bb E}[z(T)].
\end{align}
To the best of our knowledge, a general weighted trace-inverse criterion of \eqref{eq:J_def} on the full-state Gramian is not standard in the network controllability literature. Existing formulations have more commonly taken either isotropic forms \cite{pasqualetti2014,summers2015submodularity} or target/output-restricted forms based on output Gramians \cite{Klickstein2017,sato2025tcs}.

Decomposition \eqref{eq:J_decomp} provides a useful interpretation.
The first term
$\mathrm{tr}\left(W(p,T)^{-1}\mathrm{Cov}(z(T))\right)$
represents the contribution of the covariance structure of \(z(T)\).
Indeed, if
$\mathrm{Cov}(z(T))=\sum_{i=1}^n \lambda_i v_i v_i^\top$ is an eigenvalue decomposition with \(\lambda_i\ge 0\), then $\mathrm{tr}\left(W(p,T)^{-1}\mathrm{Cov}(z(T))\right) = \sum_{i=1}^n \lambda_i\, v_i^\top W(p,T)^{-1} v_i$.
Thus, this term can be interpreted as a weighted sum of directional minimum-energy costs, with weights determined by the variances along the principal directions of the displacement distribution.
The second term
$\mathbb{E}[z(T)]^\top W(p,T)^{-1}\mathbb{E}[z(T)]$
is the minimum energy required to achieve the mean displacement itself.
Hence, $J(p;T)$ evaluates the expected minimum control energy associated with the prescribed displacement distribution, decomposing it into a mean component and a variability component.

The following example serves as the basis of the brain-network experiment in Section~\ref{sec:numerical_brain}.

\begin{example} \label{ex1}
A particularly natural and practically important special case is that the terminal state \(x_T\) is prescribed deterministically, whereas the initial state \(x_0\) is random with mean \(\mu_0\) and covariance \(\Sigma_0\):
The initial state is often not known exactly and is therefore modeled probabilistically,
whereas the terminal state is specified as a desired target to be achieved.
In this case,
$\mathbb{E}[z(T)]
=
x_T-\exp(AT)\mu_0$ and 
$\mathrm{Cov}(z(T))
=
\exp(AT)\Sigma_0\exp(A^\top T)$.
\end{example}

The following example shows that
$J(p;T)$ in \eqref{eq:J_def} can be viewed as a natural generalization of a controllability measure based on the inverse controllability Gramian. 

\begin{example}
\label{ex:aecs_special}
Suppose that $x_0=0$ deterministically and that $x_T$ is uniformly distributed on the unit sphere.
Then $z(T)=x_T$, and by rotational symmetry we have
$\mathbb E[x_T]=0$ and
$\mathbb E[x_Tx_T^\top]=\frac1n I_n$.
Hence
$M(T)=\mathbb E[z(T)z(T)^\top]=\frac1n I_n$,
and therefore
$J(p;T)
=
\frac1n \mathrm{tr}\!\left(W(p,T)^{-1}\right)$.
The same optimizer is also obtained when $x_T\sim\mathcal N(0,I_n)$, in which case $M(T)=I_n$.
The function ${\rm tr}(W(p,T)^{-1})$ is a classical controllability metric used in\cite{muller1972, summers2015submodularity}.
\end{example}

\begin{remark}
The present framework does not require the full joint law of $(x_0,x_T)$ to be specified.
Indeed, in light of \eqref{eq:z_def} and \eqref{eq:M_def}, it is sufficient to specify the means and covariance matrices of $x_0$ and $x_T$, together with their cross-covariance.
Therefore, the weighting matrix $M(T)$ can be constructed without explicitly characterizing a complete probabilistic model of $(x_0,x_T)$.
This feature broadens the range of practical applications of the framework.
\end{remark}

\section{Weighted Average Energy Controllability Scoring}

For a fixed horizon \(T>0\), we define the weighted average energy controllability score via the optimization problem
\begin{framed}
\begin{equation}
\begin{aligned}
    &\text{minimize}
    && J(p;T) \\
    &\text{subject to}
    && p \in X_T \cap \Delta_n.
\end{aligned}
\label{eq:opt_problem}
\end{equation}
\vspace{-1em}
\end{framed}
\noindent
Here,
$X_T:=\{p\in\mathbb{R}^n \mid W(p,T)\succ O\}$,
and $J(p;T)$ is defined in \eqref{eq:J_def}, or equivalently by \eqref{eq:J_decomp}.
The constraint $p\in X_T$ ensures that  virtual system \eqref{eq:virtual_system} is controllable, while $p\in\Delta_n$ encodes a relative allocation of intervention resources under a fixed total budget.
Note that the feasible set $X_T\cap \Delta_n$ is nonempty.
Indeed, we have
$
W({\bf 1}/n,T)
=
\frac1n\int_0^T \exp(At) \exp(A^\top t)\,{\rm d}t \succ O$,
and thus ${\bf 1}/n\in X_T\cap\Delta_n$.

Any optimizer of Problem~\eqref{eq:opt_problem} is called a weighted average energy controllability score (W-AECS). It is interpreted as a node-wise, task-dependent intervention-importance profile: a larger value of \(p_i\) assigns a larger share of the virtual actuation budget to node \(i\), indicating that intervention at that node is expected to be more effective for realizing the task-dependent displacement geometry encoded by \(M(T)\) over the horizon \([0,T]\).

Problem~\eqref{eq:opt_problem} is related to approximate A-optimal experimental design (OED): the affine model \eqref{eq:Wp_linear} is analogous to the information-matrix model in approximate OED, and W-AECS corresponds to A-optimality, following the same line of reasoning as in \cite[Section~III]{sato2026OEDrelation}. However, unlike approximate OED \cite{harman2024polytope}, controllability scoring interprets the optimizer itself as a node-wise score, so uniqueness is essential. Although uniqueness of the standard VCS/AECS has been studied in \cite[Section~III-A]{sato2025uniqueness}, the argument below is different and is tailored to the weighted objective \(J(\cdot;T)\). 

\begin{remark}
The objective $J(p;T)$ is closely related to the standard AECS objective as shown in Example~\ref{ex:aecs_special}. In more detail, since
$M(T) \preceq \lambda_{\max}(M(T)) I_n$,
we have
\begin{equation}
J(p;T)
\leq
\lambda_{\max}(M(T))\,\mathrm{tr}\left(W(p,T)^{-1}\right)
\label{eq:J_AECS_bounds}
\end{equation}
for every feasible allocation \(p\in X_T\cap\Delta_n\).
Thus, the standard AECS objective \(\mathrm{tr}(W(p,T)^{-1})\) provides a natural baseline for evaluating \(J(p;T)\).
However, unless \(M(T)\) is proportional to the identity, the optimizer of \(J(p;T)\) need not coincide with that of the standard AECS objective. In Section~\ref{sec:numerical_brain}, we show numerically that task-dependent choices of \(M(T)\) can lead to W-AECS scorings that differ substantially from those of the standard AECS.
\end{remark}

\subsection{Existence of an optimal solution}

We first establish existence of an optimal solution to Problem~\eqref{eq:opt_problem}. Although \(\Delta_n\) is compact, the feasible set \(X_T\cap\Delta_n\) is not necessarily closed. The following lemma shows that, when \(M(T)\succ O\), the objective blows up near the singular boundary of \(X_T\), and hence the minimum is attained.

\begin{lemma}
\label{lem:existence}
Fix \(T>0\) and assume that \(M(T)\succ O\).
Then Problem~\eqref{eq:opt_problem} admits an optimal solution.
\end{lemma}
\begin{proof}
Take any $\hat p\in X_T\cap \Delta_n$ and consider the sublevel set
$S:=
\{p\in X_T\cap\Delta_n \mid J(p;T)\le J(\hat p;T)\}$.

We prove $S$ is compact in $X_T\cap\Delta_n$.
Because $\Delta_n$ is compact, it suffices to show that $S$ is closed in $\Delta_n$.
Let $\{p_k\}\subset S$ and $p_k\to \bar p\in\Delta_n$.
Suppose $\bar p\notin X_T$.
Then $W(\bar p,T)\succeq O$ is singular, while $W(p_k,T)\succ O$ and
$W(p_k,T)\to W(\bar p,T)$, so
$\lambda_{\min}(W(p_k,T))\to 0$.
Since $M(T)\succ O$, there exists $m>0$ such that
$M(T)\succeq m I_n$.
Hence
$J(p_k;T)
=
{\rm tr} \left(W(p_k,T)^{-1}M(T)\right) \geq
m {\rm tr} \left(W(p_k,T)^{-1}\right)
\geq
m \frac{1}{\lambda_{\min}(W(p_k,T))} \to \infty$,
a contradiction to $J(p_k;T)\le J(\hat p;T)$.
Thus $\bar p\in X_T$.
By continuity of $J(\cdot;T)$ on $X_T$, we also have
$J(\bar p;T)=\lim_{k\to\infty}J(p_k;T)\le J(\hat p;T)$,
so $\bar p\in S$.
Therefore \(S\) is compact in $X_T\cap\Delta_n$. 

Since \(J(\cdot;T)\) is continuous on \(X_T\), it is continuous on \(S\subset X_T\). Hence, by the Weierstrass theorem, \(J(\cdot;T)\) attains its minimum on \(S\) \cite[Proposition~A.8]{bertsekas2016nonlinear}. Moreover, by the definition of \(S\), minimizing \(J(\cdot;T)\) over \(S\) is equivalent to minimizing it over \(X_T\cap\Delta_n\). Thus Problem~\eqref{eq:opt_problem} admits an optimal solution.
\end{proof}

\subsection{Strict convexity}

To clarify the strict convexity of \(J(\cdot;T)\), we examine how the controllability Gramian varies with the allocation vector \(p\).
For a direction \(d\in\mathbb{R}^n\), define
$\Delta W(d):=\sum_{i=1}^n d_i W_i(T)$.
Since \(W(p,T)\) depends linearly on \(p\) in \eqref{eq:Wp_linear}, its directional derivative along \(d\) is given by
${\rm D}W(p,T)[d]=\Delta W(d)$.

\begin{lemma}[Directional derivatives]
\label{prop:derivatives}
Fix $T>0$ and let $p\in X_T\cap\Delta_n$, and write $W:=W(p,T)$.
Then
\begin{align}
&{\rm D}J(p;T)[d]
=
-\mathrm{tr} \left(W^{-1}\Delta W(d)\,W^{-1}M(T) \right),
\label{eq:DJ}
\\
&{\rm D}^2J(p;T)[d,d] \label{eq:D2J} \\
&=
2\mathrm{tr}\left(W^{-1}\Delta W(d) W^{-1}\Delta W(d) W^{-1}M(T)\right).
\end{align}
\end{lemma}
\begin{proof}
Eq.~\eqref{eq:DJ} is obtained by applying the identity
${\rm D}(W^{-1})[H]=-W^{-1}HW^{-1}$
to \(W=W(p,T)\) with \(H=\Delta W(d)\).
Differentiating once more in the same direction yields
 \eqref{eq:D2J}. 
 \end{proof}

Using Lemma~\ref{prop:derivatives}, we obtain the strict convexity of $J(\cdot; T)$ under the following assumption.

\begin{assumption} \label{injective_con}
    $\Delta W(d)\neq O$ for every nonzero tangent direction $d\in\mathbb{R}^n$ satisfying $\mathbf{1}^\top d=0$.
\end{assumption}

\begin{theorem}
\label{thm_unique}
Suppose that $M(T)\succ O$ and Assumption~\ref{injective_con} hold.
Then the objective function $J(\cdot;T)$ in Problem~\eqref{eq:opt_problem} is strictly convex on $X_T\cap\Delta_n$. \end{theorem}
\begin{proof}
The argument is analogous to the analysis in \cite[Section III-A]{sato2025uniqueness}, but here it is more direct to use the second directional derivative \eqref{eq:D2J} of $J(\cdot;T)$.

Let $p,q\in X_T\cap\Delta_n$ with $p\neq q$, and define
\[
\phi(t):=J((1-t)p+tq;T), \qquad t\in[0,1].
\]
Since $X_T\cap\Delta_n$ is convex, the entire segment $(1-t)p+tq$ is feasible for all $t\in[0,1]$.

Set
$d:=q-p$.
Then $d\neq 0$ and
$\mathbf{1}^\top d=\mathbf{1}^\top q-\mathbf{1}^\top p=0$,
and thus Assumption~\ref{injective_con} yields
$\Delta W(d)=\sum_{i=1}^n d_i W_i(T)\neq O$.

For each $t\in[0,1]$, write
$W(t):=W((1-t)p+tq,T)\succ O$.
By Lemma~\ref{prop:derivatives},
\begin{align}
&\phi''(t) \label{eq:phi_second} \\
&=
2\operatorname{tr}\!\left(
W(t)^{-1}\Delta W(d)\,W(t)^{-1}\Delta W(d)\,W(t)^{-1}M(T)
\right).
\end{align}
Now define
$A_t :=W(t)^{-1/2}\Delta W(d)\,W(t)^{-1/2}$ and 
$B_t :=W(t)^{-1/2}M(T)W(t)^{-1/2}$.
Then $A_t$ is symmetric and $B_t\succ O$ because $M(T)\succ O$ and $W(t)\succ O$. Hence
$\phi''(t)=2\operatorname{tr}(A_t^2B_t) = 2{\rm tr} \left( B_t^{1/2}A_t^2B_t^{1/2}\right)$.
Since $\Delta W(d)\neq O$, we have $A_t\neq O$, so $A_t^2\succeq O$ is nonzero. Because $B_t\succ O$, the matrix
$B_t^{1/2}A_t^2B_t^{1/2}$
is positive semidefinite and nonzero, and therefore has strictly positive trace. Thus
$\phi''(t)>0$
for all $t\in[0,1]$.

Hence $\phi$ is strictly convex on $[0,1]$, and thus $J(\cdot;T)$ is strictly convex on $X_T\cap\Delta_n$.
\end{proof}

The assumption \(M(T)\succ O\) in Theorem~\ref{thm_unique} is not restrictive and is satisfied in many relevant settings. For instance, in the setting of Example~\ref{ex1}, if 
$\Sigma_0\succ O$, then $M(T)\succ O$ holds for all $T>0$.

Assumption~\ref{injective_con} can be characterized more concretely. 
Indeed, for any $d\in\mathbb{R}^n$ satisfying $\mathbf{1}^\top d=0$, we have
$d_n=-\sum_{i=1}^{n-1}d_i$,
and hence
\[
\Delta W(d)
=
\sum_{i=1}^n d_i W_i(T)
=
\sum_{i=1}^{n-1} d_i\bigl(W_i(T)-W_n(T)\bigr).
\]
Therefore, Assumption~\ref{injective_con}
is equivalent to saying that the only coefficients
$(d_1,\dots,d_{n-1})$ satisfying
$\sum_{i=1}^{n-1} d_i\left(W_i(T)-W_n(T)\right)=O$    
are $d_1=\cdots=d_{n-1}=0$.
This is precisely the linear independence of
\[
W_1(T)-W_n(T),\dots,W_{n-1}(T)-W_n(T).
\]
Using this fact, we obtain the following theorem.

\begin{theorem}
\label{thm:generic_injectivity}
For any matrix $A\in\mathbb{R}^{n\times n}$, Assumption~\ref{injective_con} holds for almost all $T>0$.
\end{theorem}

\begin{proof}
Define the Gram matrix
$G(T) = (G_{ij}(T))\in\mathbb{R}^{(n-1)\times(n-1)}$ as
\begin{align}
G_{ij}(T)
:=
\mathrm{tr} \left(
\left(W_i(T)-W_n(T)\right)\left(W_j(T)-W_n(T)\right)
\right).
\end{align}
Then Assumption~\ref{injective_con} is equivalent to
\begin{align}
\det G(T)\neq 0, \label{det_G}
\end{align}
because
$c^\top G(T) c = \left\| \sum_{i=1}^{n-1} c_i (W_i(T)-W_n(T))\right\|^2_{\rm F}$ for any $c\in {\bb R}^{n-1}$, where $\|\cdot\|_{\rm F}$ denotes the Frobenius norm.

Since $W_i(T)$ is defined as \eqref{eq:Wi_def},
each entry of $W_i(T)$ is real analytic in $T$, and hence so is each entry of $G(T)$. Therefore, $\det G(T)$ is real analytic in $T$.
Moreover, we have
\begin{align}
W_i'(T)=\exp(AT)e_ie_i^\top \exp(A^\top T),
\quad
W_i(0)=O.    
\end{align}
Hence 
$W_i'(0)=e_ie_i^\top=:E_{ii}$,
and therefore as $T\downarrow 0$,
$W_i(T)=T E_{ii}+o(T)$,
which implies that
\[
\frac{1}{T}\bigl(W_i(T)-W_n(T)\bigr)\to E_{ii}-E_{nn}
\qquad (T\downarrow 0).
\]
Therefore, as $T\downarrow 0$,
\begin{align}
    \frac{1}{T^2}G_{ij}(T)
&=
\mathrm{tr} \left(
\frac{W_i(T)-W_n(T)}{T}\,
\frac{W_j(T)-W_n(T)}{T}
\right) \\
&\to
\mathrm{tr} \left((E_{ii}-E_{nn})(E_{jj}-E_{nn})\right) =: G^0_{ij}.
\end{align}
where $G^0:=(G^0_{ij})$ is the Gram matrix of the linear independent matrices
$E_{11}-E_{nn},\dots,E_{n-1,n-1}-E_{nn}$.
Hence $G^0$ is positive definite, and in particular $\det G^0\neq 0$. By continuity of the determinant,
$\det \left(\frac{1}{T^2}G(T)\right)\neq 0$
for all sufficiently small $T>0$. Equivalently, \eqref{det_G} holds for all sufficiently small $T>0$. Thus $\det G(T)$ is not identically zero.
By the standard zero-set property of real analytic functions, its zero set has Lebesgue measure zero \cite{mityagin2020zero}.
This completes the proof.
\end{proof}

\subsection{Uniqueness}

Theorem~\ref{thm:generic_injectivity} shows that Assumption~\ref{injective_con} is generically satisfied for almost all $T>0$.
Combining this fact with Lemma~\ref{lem:existence} and Theorem~\ref{thm_unique}, we obtain generic uniqueness of the optimizer whenever $M(T)\succ O$.

\begin{corollary} \label{cor:generic_unique} 
Suppose that $M(T)$ in \eqref{eq_M} is positive definite.
    Then, for any matrix $A\in {\bb R}^{n\times n}$ and almost all $T>0$, Problem~\eqref{eq:opt_problem} admits a unique optimal solution.
\end{corollary}
\begin{proof}
By Theorem~\ref{thm:generic_injectivity}, Assumption~\ref{injective_con} holds for almost all \(T>0\). Fix such a \(T>0\). Then Lemma~\ref{lem:existence} guarantees existence of an optimal solution to Problem~\eqref{eq:opt_problem}.
By Theorem~\ref{thm_unique}, \(J(\cdot;T)\) is strictly convex on \(X_T\cap\Delta_n\). Therefore, the optimal solution is unique \cite[Proposition~1.1.2]{bertsekas2016nonlinear}.
\end{proof}

Corollary~\ref{cor:generic_unique} guarantees the uniqueness of W-AECS.
This uniqueness is essential for interpreting W-AECS as a task-dependent centrality measure for network system \eqref{eq:autonomous}, because it ensures that each node is assigned a well-defined score that can be meaningfully compared across nodes.
Without uniqueness, multiple optimal allocations could yield different score assignments, making interpretation and comparison ambiguous.

\subsection{Optimization algorithm for W-AECS}

Optimization problem~\eqref{eq:opt_problem} can be solved in essentially the same manner as the projected-gradient schemes used for VCS and AECS in \cite{sato2022controllability,sato2025uniqueness}.
For completeness, we briefly summarize a projected-gradient method to obtain W-AECS. 

In the projected-gradient method, we consider the following iteration:
$p^{(k+1)}
=
\Pi_{\Delta_n} \left(
p^{(k)}-\alpha_k \nabla J(p^{(k)};T)
\right)$,
where \(\Pi_{\Delta_n}\) denotes the Euclidean projection onto the simplex
\(\Delta_n\), and the step size \(\alpha_k>0\) is chosen by a backtracking
Armijo rule along the projection arc.
Here, $\nabla J(p;T):=
\left(
\frac{\partial J}{\partial p_i}(p;T)
\right)$
denotes
the gradient
of \(J(\cdot;T)\), which is given by
\begin{align}
\frac{\partial J}{\partial p_i}(p;T)
&= {\rm D}J(p;T)[e_i] \\
&=-\mathrm{tr} \left(
W(p,T)^{-1}W_i(T)\,W(p,T)^{-1}M(T)
\right).
\label{eq:grad_i}
\end{align}

This projected-gradient scheme provides a practical method for computing W-AECS.
Any convergent subsequence of the generated sequence has an optimal solution of Problem~\eqref{eq:opt_problem} as its limit.
Moreover, by Corollary~\ref{cor:generic_unique}, for almost all $T>0$ the optimizer is unique whenever $M(T)\succ O$.
Hence, in such cases, the limit is uniquely determined and coincides with W-AECS.
A more detailed convergence analysis can be carried out in essentially the same manner as in \cite[Section~III-D]{sato2025uniqueness}.

\section{Numerical experiment on brain networks}
\label{sec:numerical_brain}

In this section, we present a numerical experiment using the same brain-network dataset as that used in \cite[Section~IV]{sato2025uniqueness}.
The dataset is publicly available online\footnote{Available: \url{https://doi.org/10.17605/OSF.IO/YW5VF}} and is described in \cite{Skoch2022}.

\subsection{Purpose of the experiment}
\label{subsec:purpose_brain}

A central motivation for the present experiment is the marked discrepancy between the standard average energy controllability score (AECS) and the volumetric controllability score (VCS) observed in \cite[Section~IV]{sato2025uniqueness}.
In that study, AECS tended to assign high importance to brain regions associated mainly with cognitive and motor functions, whereas VCS tended to highlight regions related more strongly to sensory and emotional processing.
Therefore, they exhibit qualitatively different evaluation tendencies on the same brain-network dataset.

Unlike the standard AECS, the proposed W-AECS depends explicitly on the state transition to be emphasized through the weighting matrix \(M(T)\).
This suggests the possibility that, by choosing a suitable transition of interest, W-AECS may produce a scoring substantially different from that of the isotropic AECS.
A particularly interesting question is whether W-AECS can assign relatively high importance to regions that were previously emphasized by VCS.

Accordingly, the purpose of the present experiment is not simply to obtain another controllability-based scoring, but to examine whether a transition-dependent energy-based score can partially reproduce a VCS-like evaluation tendency, despite the fact that standard AECS and VCS are known to behave very differently.
If this occurs, it would indicate that the controllability-based importance of brain regions depends not only on the underlying network structure and on the choice of scalar criterion, but also crucially on which state-transition directions are regarded as important from the viewpoint of intervention.

\subsection{Dataset, network dynamics, and time horizon}
\label{subsec:dataset_dynamics_brain}

We consider the same structural brain-network dataset as in \cite[Section~IV]{sato2025uniqueness}.
For each subject, the dataset provides a directed weighted connectivity matrix over \(n=90\) brain regions of interest (ROIs), based on the AAL atlas.
From this connectivity matrix, we construct the corresponding graph Laplacian \(L\), and define the network dynamics by
\begin{equation}
\dot{x}(t)=Ax(t),
\qquad
A=-L.
\label{eq:brain_dynamics}
\end{equation}

\noindent
Here, \(x(t)\in\mathbb{R}^{90}\) is interpreted as a coarse macroscopic brain state over the 90 ROIs.
It is not assumed to coincide directly with an observed BOLD time series, but is instead regarded as an abstract dynamical variable associated with the structural brain network.
This abstraction is appropriate because our aim is not to reproduce measured time series themselves, but to examine how controllability-based intervention scores depend on the direction and weighting of prescribed state transitions on the network.
This abstraction is also supported by prior work on macroscale brain-network modeling: Abdelnour \emph{et al.} \cite{abdelnour2014network} showed that a linear network-diffusion model defined on the structural connectivity graph can successfully predict resting-state functional connectivity, in some cases more accurately than several alternative models. This suggests that, even if microscopic neural activity is highly nonlinear, a linear graph-Laplacian-type description can still be meaningful at the coarse ROI level considered here.

Because W-AECS is inherently a finite-horizon quantity, the time horizon \(T\) must be specified explicitly.
To remain consistent with \cite[Section~IV]{sato2025uniqueness}, we used the same reference horizon
$T=100$.

\subsection{Task specification: transition from a rest-related state to a VCS-favored state}
\label{subsec:task_specification_brain}

A key step in the proposed experiment is to define an intervention task that is both interpretable and capable of producing a non-isotropic weighting structure.
Rather than adopting the isotropic choice \(M(T)=I_{90}\), we construct \(M(T)\) from uncertainty in the initial state together with a prescribed terminal target.
The task considered here is designed so that the prescribed transition emphasizes brain regions that are closer to those favored by VCS in \cite{sato2025uniqueness}, rather than those typically emphasized by the standard isotropic AECS.
More specifically, we construct a transition from a rest-related proxy state toward a VCS-favored proxy state.

\subsubsection{ROI classification used in the experiment}
\label{subsubsec:roi_classification_brain}

Because the AAL atlas is anatomical rather than functional, the classification adopted here should be understood as an interpretable proxy for rest-related and VCS-favored subnetworks, rather than as an exact functional partition.

The rest-related set is defined as
\begin{equation}
S_{\mathrm{rest}}
=
\{23,24,35,36,65,66\},
\label{eq:Srest}
\end{equation}
corresponding to
\begin{itemize}
\item 23,\,24: Left/Right Superior Medial Gyrus,
\item 35,\,36: Left/Right Posterior Cingulum,
\item 65,\,66: Left/Right Angular Gyrus.
\end{itemize}
These AAL regions were not chosen because they provide an exact functional parcellation of the default mode network, but because they offer anatomically interpretable proxies for three canonical DMN components frequently identified in the literature: a medial prefrontal component, a posterior midline component centered on the posterior cingulate/precuneus, and a lateral parietal component including the angular gyrus \cite{buckner2008brain,andrews2014default,menon2023dmn}.

The target set is defined as
\begin{equation}
S_{\mathrm{VCS}}
=
\{22,41,42,79,80\},
\label{eq:SVCS}
\end{equation}
corresponding to
\begin{itemize}
\item 22: Right Olfactory Cortex,
\item 41,\,42: Left/Right Amygdala,
\item 79,\,80: Left/Right Heschl's Gyrus.
\end{itemize}
These regions form a compact proxy for the sensory/emotional set emphasized by VCS in \cite[Table~V]{sato2025uniqueness}.
Thus, the target set is chosen not to mimic the regions favored by the standard AECS, but to test whether a suitably designed transition can induce a W-AECS ranking closer to that of VCS.

\subsubsection{Construction of the initial and terminal states}
\label{subsubsec:task_construction_brain}

Based on these ROI sets, we considered a random initial state and a deterministic terminal state of the form
\begin{equation}
x_0 \sim \mathcal{N}(\mu_0,\Sigma_0),
\qquad
x_T \in \mathbb{R}^{90},
\label{eq:x0_random_xT_deterministic_brain}
\end{equation}
where \(\mu_0\), \(\Sigma_0\), and \(x_T\) are specified below.
This asymmetric specification is natural in intervention problems: the initial state is often not known exactly and is therefore modeled probabilistically, whereas the terminal state is prescribed as a desired target to be achieved.
In this case, \({\bb E}[z(T)]\) and \({\rm Cov}(z(T))\) are given in Example~\ref{ex1}.

We defined the initial mean vector \(\mu_0\in\mathbb{R}^{90}\) and the prescribed target state \(x_T\in\mathbb{R}^{90}\) by assigning positive entries to the selected ROI sets and zero elsewhere.
The initial mean \(\mu_0\) was defined by
\begin{equation}
(\mu_0)_j
=
\begin{cases}
+1, & j\in S_{\mathrm{rest}},\\
0, & \text{otherwise},
\end{cases}
\label{eq:mu0_def_brain}
\end{equation}
whereas the target state \(x_T\) was defined by
\begin{equation}
(x_T)_j
=
\begin{cases}
+1, & j\in S_{\mathrm{VCS}},\\
0, & \text{otherwise}.
\end{cases}
\label{eq:xT_def_brain}
\end{equation}
This choice provides a simple and interpretable proxy for a structured transition: the initial mean emphasizes the rest-related set, whereas the target state emphasizes the VCS-favored set.
Accordingly, the vectors \eqref{eq:mu0_def_brain} and \eqref{eq:xT_def_brain} define a transition geometry that moves weight from the rest-related proxy set toward the VCS-favored proxy set.

\subsubsection{Construction of the covariance matrix}
\label{subsubsec:covariance_brain}

As a baseline, we chose an isotropic initial covariance
$\Sigma_0=0.01 I_{90}$.
This yields a simple and transparent baseline in which uncertainty enters only through the initial condition, while the directional structure of \(M(T)\) in \eqref{eq_M} is determined primarily by the mean displacement \({\bb E}[z(T)]\).
This choice is intentional in the present experiment.
By keeping the covariance isotropic and small, we isolate as much as possible the effect of the prescribed transition direction itself.
Hence, if W-AECS exhibits a ranking pattern substantially different from that of the standard AECS, such a difference can be interpreted primarily as an effect of transition-dependent weighting rather than as an artifact of anisotropic uncertainty modeling.

\subsection{Results}
Fig.~\ref{fig:WAECS_T100} illustrates the ROI-wise distribution of W-AECS across 88 individuals at \(T=100\) by showing box plots for the top 5 and bottom 5 ROIs, where the ranking is determined by the average W-AECS values over the 88 individuals.
The top 5 ROIs coincide exactly with the prescribed set \(S_{\mathrm{VCS}}\), indicating that the task-dependent weighting can make the highest-ranked W-AECS nodes resemble those favored by VCS.
However, this resemblance is only partial.
The bottom 5 ROIs are completely different from those of VCS, showing that the overall ranking structure of W-AECS remains substantially different from that of VCS when all ROIs are taken into account.
Moreover, the W-AECS values of the bottom 5 ROIs are much smaller than those of the top 5 ROIs, revealing a pronounced separation between highly emphasized and weakly emphasized nodes.
Such a marked contrast between the top and bottom scores was not observed for either VCS or AECS in the corresponding experiments reported in \cite[Section~IV]{sato2025uniqueness}, and is therefore a distinctive feature of W-AECS in this setting.

\begin{figure}[t]
   \centering
   \begin{minipage}[b]{0.49\linewidth}
     \centering
     \includegraphics[width=\linewidth]{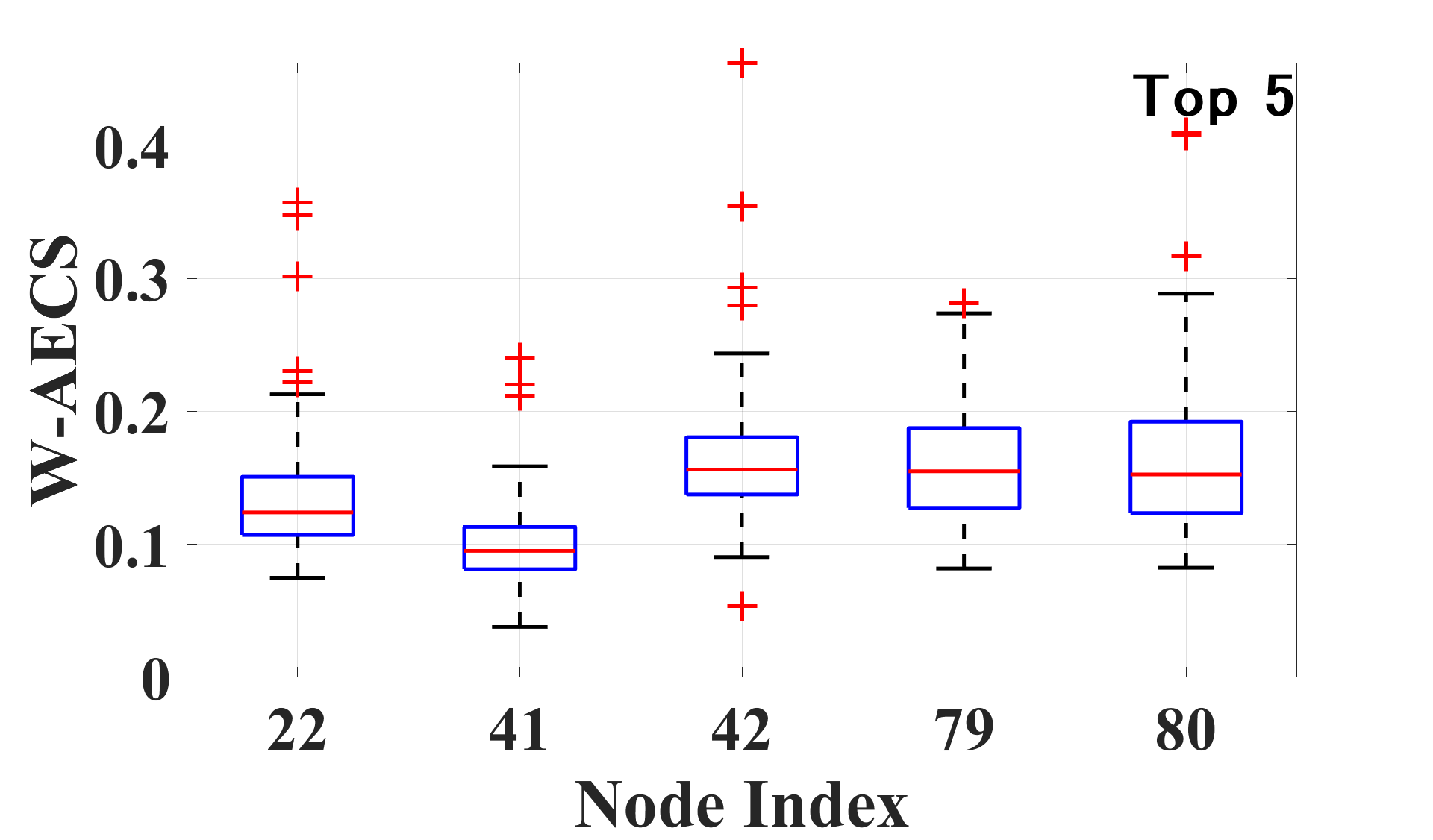}
   \end{minipage}
   \hfill
   \begin{minipage}[b]{0.49\linewidth}
     \centering
     \includegraphics[width=\linewidth]{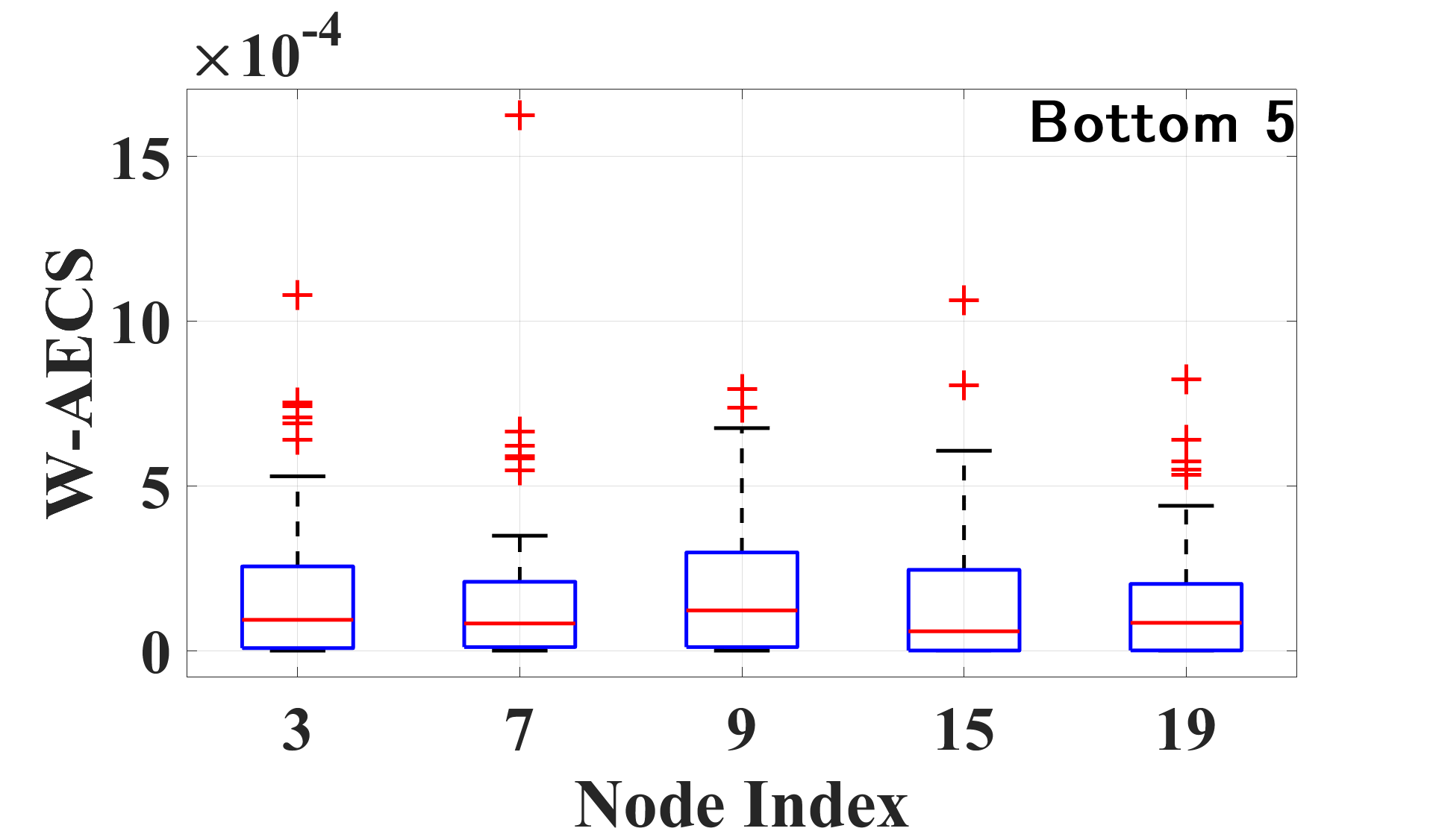}
   \end{minipage}
   \caption{Box plots of W-AECS for $T=100$. (Left: Top 5; Right: Bottom 5.)}
   \label{fig:WAECS_T100}
\end{figure}

Moreover, Table~\ref{tab:corr_waecs} summarizes the Pearson correlations between W-AECS and the standard scores AECS and VCS.
The AECS and VCS values used in this comparison were taken from the results on the same brain-network dataset reported in \cite[Section~IV]{sato2025uniqueness}.
The mean Pearson correlation between W-AECS and VCS was positive but only moderate, whereas that between W-AECS and AECS was negative.
Thus, although the transition-dependent weighting could reshape the very top of the ranking so that W-AECS selected VCS-like ROIs, it did not make W-AECS globally equivalent to VCS over all ROIs.
Rather, these results suggest that W-AECS should be understood as a distinct transition-dependent score: it can reproduce a VCS-like preference at the level of the highest-ranked ROIs, while still maintaining an overall scoring structure different from both the standard AECS and VCS.

\begin{table}[t]
\centering
\caption{Pearson correlations of W-AECS with AECS and VCS over all ROIs.}
\label{tab:corr_waecs}
\begin{tabular}{lcc}
\hline
Comparison & Mean & Std. \\
\hline
W-AECS vs AECS & $-0.251417$ & $0.054017$ \\
W-AECS vs VCS  & $0.547227$  & $0.095398$ \\
\hline
\end{tabular}
\end{table}

\section{Conclusions}

\subsection{Summary}
In this paper, we proposed the weighted average energy controllability score (W-AECS) as a task-dependent extension of the standard average energy controllability score.
Unlike the isotropic formulation of AECS, the proposed framework allows the score to reflect a prescribed transition of interest through a weighting matrix derived from the underlying control task.
This provides a control-theoretic way to evaluate intervention importance when different state-transition directions should not be treated equally.

We also studied the optimization structure of the proposed formulation and showed that it admits a well-posed interpretation as a node-wise score.
In particular, we established convexity properties of the objective and proved a generic uniqueness result with respect to the time horizon.
These results are important because they ensure that the proposed score is not merely an optimizer of an auxiliary allocation problem, but can be interpreted as a meaningful and essentially unique intervention-importance profile.

To illustrate the proposed framework, we conducted a numerical experiment on a structural brain-network dataset.
The experiment was designed to examine whether transition-dependent weighting can alter the ranking structure in a nontrivial way, especially in a direction closer to the volumetric controllability score (VCS), which was previously observed to behave quite differently from the standard AECS on the same dataset.
The results showed that, when the weighting matrix was constructed from a transition toward a VCS-favored target state, the top-ranked ROIs under W-AECS became clearly VCS-like.
At the same time, the overall ranking remained globally different from both the standard AECS and VCS.
This indicates that W-AECS should be understood not as a replacement for either standard score, but as a distinct task-dependent score that can reveal different intervention priorities depending on the transition to be emphasized.

\subsection{Extension}

Although the present paper focuses on linear time-invariant systems of the form \eqref{eq:autonomous}, the proposed W-AECS is not tied to time-invariant dynamics.
This concept can be naturally extended to linear time-varying systems, including temporal network models, as in \cite{umezu2026ltv}.

\subsection{Future Works}

Several directions remain for future work:
\begin{itemize}
    \item
    It is important to better understand how the weighting structure determines the resulting score and to characterize regimes in which W-AECS becomes close to, or substantially different from, existing controllability scores.

    \item
    Following the recent scaled-Gramian approach for infinite-horizon controllability scores in \cite{umezu2026infinite}, it would be of interest to develop an analogous infinite-horizon theory for W-AECS together with a numerically stable computational method.
    Such a formulation may allow the computation of W-AECS through matrix equations such as Lyapunov- or Sylvester-type equations, rather than through direct integration over a long time horizon.

    \item
    More systematic comparisons on brain networks and other complex systems, including ranking similarity, robustness to task specification, and sensitivity to modeling choices, would help clarify the practical value of transition-dependent controllability scoring.
\end{itemize}




\end{document}